\documentclass[final,leqno,onefignum,onetabnum]{siamltex1213}
\usepackage{amsmath}
\usepackage{verbatim}
\title{On weakly sequentially complete Banach spaces.} 

\author{Ewa M. Bednarczuk$^{1, 2}$,  Krzysztof Le\'sniewski$^2$ \thanks{$1$ Systems Research Institute, Polish Academy of Sciences, ul.  Newelska 6, 01-447 Warsaw, Poland, $2$ Warsaw University of Technology, Faculty of Mathematics and Information Science, ul. Koszykowa 75,
 00-662 Warsaw, Poland}, 
(\email{email addresses: e.bednarczuk@mini.pw.edu.pl}, \email{k.lesniewski@mini.pw.edu.pl}). }

\begin{document}
\maketitle
\slugger{mms}{xxxx}{xx}{x}{x--x}

\begin{abstract}
We provide sufficient conditions for a Banach space $Y$ to be weakly sequentially complete. These conditions are expressed in terms of the existence of directional derivatives for cone convex mappings with values in $Y$.

\end{abstract}



\pagestyle{myheadings}
\thispagestyle{plain}
\markboth{\sc On the existence of directional derivatives for cone-convex mappings}{\sc Ewa M. Bednarczuk and K. Le\'sniewski}

\section{Introduction} 
Weakly sequentially complete Banach spaces were introduced in \cite{Banach}. Since then,  properties of weakly sequentially complete spaces were investigated in a number of papers, see e.g. \cite{casini} and \cite{rosenthal} and the references therein. Important result  on weakly sequentially complete Banach spaces  is given by  Rosenthal in \cite{rosenthal}.
\begin{theorem}
If $Y$ is weakly sequentially complete, then $Y$ is either reflexive or contains a subspace isomorphic to $l^1.$
\end{theorem}

 In the present paper we prove sufficient conditions for a Banach space $Y$ to be weakly sequentially complete. These sufficient conditions are expressed in terms of existence of directional derivatives for cone convex mappings.

Cone convex mappings appear in variational analysis in topological vector spaces and in construction of efficient iterative schemes for solving vector optimization problems. Newton method for  solving smooth unconstrained vector optimization problems under partial order induced by general closed convex pointed cone can be found in \cite{drummond}.  

The organization of the paper is as follows. In Section 2 we discuss basic facts concerning weakly sequentially complete Banach spaces. In Section 3 we prove some properties of cone convex mappings. Section 4 is devoted to the constructions of convex functions that take values in given points. Section 5 contains the main result namely the proof of the fact that the existence of directional derivatives at any point and any direction for any cone convex mapping implies that the Banach space $Y$ is weakly sequentially complete.  In Section 6 we discuss properties of the mapping $F$ constructed in Theorem \ref{main}.

\section{Preliminary facts}

Let us present some known facts about weakly sequentially complete Banach spaces.

\begin{definition}
	A sequence $\{y_n\}$ in a Banach space $Y$ is weak Cauchy if $\lim\limits_{n \rightarrow \infty } y^*(y_n)$ exists for every $y^* \in Y^*.$ We say that a Banach space $Y$ is weakly sequentially complete if every weak Cauchy sequence weakly converges in $Y$. A weak Cauchy sequence $\{y_n\}$ is called nontrivial if it does not weakly converge.
	
\end{definition} 
Nontrivial weak Cauchy sequences were investigated by \cite{diestel, kadets,pelczynski}.
Any weak Cauchy sequence $\{y_n\}$ in a Banach space $Y$ is norm-bounded by the Uniform Boundedenss principle (\cite{albiac}, p. 38).
In \cite{conway} we can find a well know fact about reflexive space (\cite{conway} Corollary 4.4). 
\begin{proposition}[\cite{conway}]
	If $ Y$ is reflexive, then $Y$ is weakly sequentially complete.
\end{proposition}


However, in some nonreflexive spaces, sequences can be found which are weak Cauchy but not weakly convergent.
Let us present some examples, c.f. \cite{conway,Pei}.

\newtheorem{Example}{Example}
\begin{Example}
\begin{itemize}
\item Let $Y=c_0$ and let us consider $y_n:=\sum\limits_{k=1}^n e_k$, where $e_k$ is basic vector in $c_0.$  We have $(c_0)^*=l^1$ and  $\{y_n\}$ is weak Cauchy sequence. Since $\{y_n\}$ converges $\mbox{weak}^*$ to the element $(1,1,\dots),$  it is not weakly convergent. 

	\item
	If we take $K$ which is a compact Hausdorff space, then $C(K)$ is weakly sequentially complete if and only if $K$ is finite.
	\item
	Space $C[0,1]$ of all continuous functions on $[0,1]$ is not weakly sequentially complete. 
	\item
	For every measure $\mu$, the space $L^1(\mu)$ is weakly sequentially complete.
\end{itemize}
\end{Example}

\begin{proposition}
	Any closed subspace $Z\subset Y$ of a weakly sequentially complete Banach space $Y$ is weakly sequentially complete Banach space.
\end{proposition}
\begin{proof}
	Consider a weak Cauchy sequence $\{y_n\}$ in $Z$. Then it is weak Cauchy in $Y$. Since $Y$ is weakly sequentially complete,  $\{y_n\}$ weakly converges to some $y\in Y.$ Thus, $y$ lies in the weak closure of $Z$. Since $Z$ is closed and convex, by Mazur's theorem, weak closure of $Z$ coincides with $Z.$ Hence $y\in Z.$
\end{proof}

\begin{proposition}
	\label{podciag}
	Let $\{y_n\}\in Y$ be a nontrivial weak Cauchy sequence. Then each subsequence of $\{y_n\}$ is also nontrivial weak Cauchy.
\end{proposition}
\begin{proof}
	By contradiction, assume that there exists a subsequence $\{y_{n_k}\} \subset \{y_m\}$ weakly converging to  $y_0 \in Y.$ Let $\varepsilon >0$. There exists $N_1\in \textbf{N}$ such that for all $k>N_1$ $|f^*(y_{n_k})-f^*(y_0)|\le \frac{\varepsilon}{2}.$
	Since $\{y_n\}$ is weak Cauchy,  there exists $N\in \textbf{N}$ such that  $|f^*(y_{n})- f^*(y_{n_k})|\le \frac{\varepsilon}{2}$ for all $n>N.$ Since $k, n> \mbox{max}\{N_1,N\}$ we have
	$$
	\begin{array}{l}
	|f^*(y_{n})- f^*(y_{0})|=|f^*(y_{n})- f^*(y_{n_k})+ f^*(y_{n_k})-f^*(y_0)|\le
	\\|f^*(y_{n})- f^*(y_{n_k})|+ |f^*(y_{n_k})-f^*(y_0)|\le \frac{\varepsilon}{2}+\frac{\varepsilon}{2}=\varepsilon.
	\end{array}
	$$\end{proof}



In the present paper we consider cone ordering relations in $Y$ generated by  convex closed cones. If we  assume that $Y$ is a Banach lattice, then we get the following characterization (\cite{lindenstrauss}).

\begin{theorem}[\cite{lindenstrauss}]
	If $Y$ is a Banach lattice, then $Y$ is weakly sequentially complete if and only if it does not contain any  subspace isomorphic to $c_0$. 
\end{theorem}

Let us underline that if we consider the concept of weak completeness in terms of nets  (see \cite{engelking} for the definition), then we cannot expect weak completeness in the infinite dimensional case. 
A net $S:=\{x_\sigma, \sigma \in \sum\}$ in a topological vector space $Y$ is  Cauchy if for every neighbourhood $U$ of zero there exists $\sigma_0 \in \sum $ such that $x_{\sigma_1}-x_{\sigma_2} \in U$ for every $\sigma_1, \sigma_2 \ge \sigma_0.$
Topological vector space $Y$ is complete if each Cauchy net converges (see \cite{joshi} p.356). 


First countable spaces, including metric spaces, have topologies that are determined by their convergent sequences. In an arbitrary first countable space, a point $x$ is in the closure of a subset $A$ if and only if there is a sequence in $A$ converging to $x.$

If $Y$ is an infinite dimensional Banach space, regardless of whether or not it is separable in the norm topology, then the weak topology on $Y$ is not first countable, and is not characterized by its convergent sequences alone.

For  the weak topology in an infinite-dimensional space, there exist weak Cauchy nets which are not weakly convergent. This result can be found in \cite{costara,diestel} and is deeply tied with Axiom of choice and Helly's Theorem.
\begin{proposition}[\cite{diestel} p.14]
The weak topology on infinite-dimensional normed space is never complete.
\end{proposition}

\section{Cone-convex mappings}

Let $X$, $Y$ be real linear vector  spaces 
and let $K\subset Y$ be a convex cone inducing the standard ordering relation
$$
x \le_K y  \ \ \Leftrightarrow  \ \  y-x \in K \ \ \mbox{ for\  any }\ \  x,y \in Y.
$$
Analogously,  we write $y\ge_K x$ if and only if $y-x\in K.$ We use the notation $\ge$ if $K$ is clear from the context.

Let  $F: X \rightarrow Y$. We say that the mapping $F$ is $K$-convex  on a convex set $A \subset X$ if 
\begin{equation}
\label{kwypuklosc}
F(\lambda x + (1-\lambda )y) \le_K \lambda F(x) + (1-\lambda ) F(y) \ \ \mbox{ for all } x,y \in A \ \ \mbox{and } \lambda \in [0,1].
\end{equation}

As in the scalar case \cite{zalinescu} we have the following characterization of cone convexity of $F$ in terms of the epigraph of $F$, $\mbox{epi} F$, defined as $\mbox{epi} F:=\{(x,y)\in A \times Y:\ y\ge_K F(x)\}$.
\begin{proposition}
Let $K\subset Y$ be a closed convex cone.
A mapping $F: X \rightarrow Y$ is $K$-convex on $A\subset X$ if and only if $\mbox{epi} F$ is a convex set in $A\times Y$.
\end{proposition}
\begin{proof}
Let us assume that $F$ is $K$-convex on $A$. Let $(x_1,y_1), (x_1,y_1)\in \mbox{epi} F$ and $\lambda \in [0,1].$ From the definition of  $\mbox{epi} F$ we have $y_1- F(x_1)=k_1$ and $y_2- F(x_2)=k_2$ for some $k_1, k_2 \in K.$ Since $K$ is convex,
$
\lambda k_1 + (1-\lambda)k_2 \in K
.$
From $K$-convexity of $F$ on $A$ we get 
$$
F(\lambda x_1 +(1-\lambda) x_2)\le_K \lambda F(x_1) + (1-\lambda) F(x_2)\le_K \lambda y_1+ (1-\lambda )y_2.
$$
The converse implication follows directly from the definition of $\mbox{epi} F.$
\end{proof}

Let $Y$ be a Banach space with the dual space $Y^*.$ Let $K$ be a closed convex cone in $Y.$  Dual cone $K^*$ of $K$ is defined as $K^*:=\{ y^* \in Y^* \ :\ y^*(y) \ge 0 \ \forall y\in K\}$. 

\begin{proposition}[\cite{jahn}, Lemma 3.21]
Let $K\subset Y$ be a closed convex cone in $Y.$
The cone $K$ can be represented as 
\begin{equation}
\label{biduall}
K=\{y\in Y\ : \ y^*(y)\ge 0\ \forall y^*\in K^*\}.
\end{equation}
\end{proposition}
The following lemma provides another characterization of $K$-convex mappings. The finite-dimensional case of this lemma has been used in \cite{pennanen}.

\begin{lemma}
	\label{lemma-convex}
	 Let $A\subset X$ be a convex subset of $X$. Let $K\subset Y$ be a  closed convex  cone  and let  $F:X\rightarrow Y$ be a mapping. The following conditions are equivalent.
\begin{enumerate}
\item
$
\mbox{The mapping }F \mbox{ is $K$-convex on }A.$
\item $\mbox{For any } u^{*}\in K^{*} \mbox{, the composite function }  u^{*}(F):A\rightarrow\textbf{R}
\mbox{ is convex}$.
\end{enumerate}
\end{lemma}
\begin{proof}
The implication $\Rightarrow $  follows directly  from the definition of $K^*.$

We prove the converse implication by contradiction. Suppose that 
for some $\lambda \in [0,1],$ $x,y\in A$ we have
 $ \lambda F(x) + (1-\lambda ) F(y) - F(\lambda x + (1-\lambda )y)\notin K.$ By (\ref{biduall}),  there exists $u^*\in K^*$ such that $u^*(\lambda F(x) + (1-\lambda ) F(y)-F(\lambda x + (1-\lambda )y))<0$ which, by Proposition \ref{biduall},  contradicts the convexity of $u^*(F)$ on $A$.
 \end{proof}

\section{Construction of Convex Functions}
In this section we construct convex function which take given  values at a countable number of points.
Some constructions of convex functions are present in the literature. For example, in \cite{konderla} we can find the proof that a continuous convex function $f$ defined on $l_p$, $p\ge1$,
$$
f(x):=\|(|x_1|^{1+1}, |x_2|^{1+1/2},...)\|_p,
$$
is everywhere compactly differentiable but not Fr\'echet differentiable at zero.
In \cite{bilek} we can find an extension of a function $\Phi : X \rightarrow \bar{\textbf{R}}$ to some convex function $\eta : C \rightarrow \bar{\textbf{R}}$, $X\subset C$, $C$ is a convex set,  such that $\eta (x) = \Phi(x)$ for all $x\in X.$

Let us start with the following proposition.

\begin{proposition}
\label{konstrukcja1}
For every nonnegative decreasing and convergent sequence $\{a_m\}\subset \textbf{R}$ there exists a subsequence $\{a_{m_k}\}\subset \{a_m\},$ $k\in \textbf{N}$, and a convex nonincreasing
  function $g:\textbf{R}\rightarrow \textbf{R}$ such that $g(m_k)=a_{m_k}.$
\end{proposition}
\begin{proof}
Let $a_{m_1}:= a_1$, $a_{m_2}:=a_2.$ 
Let  $f_1 :\text{R} \rightarrow \text{R}$ be defined as
$$
f_1(x):= a_{m_1}+\frac{x-1}{m_2-1}(a_{m_2}-a_{m_1}).
$$
Let $x_2$ be such that $f_1(x_2)=0.$ We have $x_2>m_2.$
By taking $m_{3}:= [x_2]+1$ where $[x]$ is the largest integer not greater than $x.$ We define $f_2 :\text{R} \rightarrow \text{R}$ as
$$
f_2(x):= a_{m_2}+\frac{x-m_2}{m_3-m_2}(a_{m_3}-a_{m_2}).
$$
Suppose that we have already defined $f_1,\dots,f_k$, $k\ge2.$ We define $f_{k+1}$ in the following way. Let  $x_k$ be such that $f_k(x_k)=0.$ We have $x_k>m_k.$ By taking $m_{k+2}=[x_k]+1$ we define
$f_{k+1} :\text{R} \rightarrow \text{R}$ as
$$
f_{k+1}(x):= a_{m_{k+1}}+\frac{x-m_{k+1}}{m_{k+2}-m_{k+1}}(a_{m_{k+2}}-a_{m_{k+1}}).
$$
We define $g:\textbf{R} \rightarrow \textbf{R}$ by formula  $g(x)=\sup\limits_k f_k(x).$ Clearly function $g$ is convex on $\textbf{R}.$ Let us show that 
\begin{equation}
\label{nierownosci-konstrukcja1}
\begin{array}{ll}
f_{k+1}(x) \ge f_k (x) & \mbox{ for } x\ge m_{k+1}\\
f_{k+1}(x) \le f_k (x) & \mbox{ for } x\le m_{k+1}.
\end{array}
\end{equation}
By construction of $m_{k+1}$ we have $a_{m_{k+2}}= f_k(m_{k+2})$ and
$$
a_{m_{k+2}}= a_{m_k} + \frac{{m_{k+2}}-{m_{k}}}{{m_{k+1}}-m_k}(a_{m_{k+1}}-a_{m_{k}})\ge a_{m_k} + \frac{{m_{k+2}}-{m_{k+1}}}{{m_{k+1}}-m_k}(a_{m_{k+1}}-a_{m_{k}})
$$
next from monotonicity of $\{a_m\}$ we get
$$
\frac{a_{m_{k+1}}-a_{m_{k+2}}}{m_{k+2}-m_k}\le \frac{a_{m_{k}}-a_{m_{k+2}}}{m_{k+2}-m_k}\le \frac{a_{m_{k}}-a_{m_{k+1}}}{m_{k+1}-m_k}
$$
and
\begin{equation}
\label{rano}
\frac{a_{m_{k}}-a_{m_{k+1}}}{m_{k+1}-m_k}\ge \frac{a_{m_{k+1}}-a_{m_{k+2}}}{m_{k+2}-m_{k+1}}
.\end{equation}
To get  (\ref{nierownosci-konstrukcja1}) we multiply above inequality by $m_{k+1}-x.$ 
Let us assume that $m_{k+1}\ge x$ then from (\ref{rano})
$$
\begin{array}{ll}
\frac{a_{m_{k}}-a_{m_{k+1}}}{m_{k+1}-m_k}(m_{k+1}-x)+\frac{a_{m_{k+2}}-a_{m_{k+1}}}{m_{k+2}-m_{k+1}}(m_{k+1}-x)&\ge 0\\
(a_{m_{k}}-a_{m_{k+1}})\left(  1 -\frac{x-m_k}{m_{k+1}-m_k}  \right)+\frac{a_{m_{k+2}}-a_{m_{k+1}}}{m_{k+2}-m_{k+1}}(m_{k+1}-x)&\ge 0\\
\end{array}
$$
we get $f_k(x)  \ge f_{k+1}(x).$ Analogous we can prove first inequality in (\ref{nierownosci-konstrukcja1}) for $x\ge m_{k+1}.$ 

By (\ref{nierownosci-konstrukcja1}), function $g$ is decreasing and $g(m_k)=a_{m_k}.$
\end{proof}

In Proposition \ref{konstrukcja1} we constructed function $g$ when $\{a_m\} $ is nonnegative, decreasing and $\{t_k\}$, $t_k=m_k$, $k\in \text{N}$ is a subsequence of integers. Now we consider the case when $\{a_m\}$ is arbitrary and $\{t_k\}$ is a decreasing sequence of reals.

\begin{lemma}
\label{lemat1}
Let  $\{a_{m}\},\{t_m\}\subset\textbf{R}$ be sequences with $\{t_m\}$ decreasing. 
If
\begin{equation}
\label{nierwonoscwaznaa}
\frac{a_{m+1}-a_{m}}{t_{m+1}-t_{m}} \ge \frac{a_{m+2}-a_{m+1}}{t_{m+2}-t_{m+1}} \ \ \ \mbox{ for } m\in\textbf{N},
\end{equation}
 there exists a convex function $g:\textbf{R}\rightarrow \textbf{R}$ such that 
$$
 g(t_{m})=a_{m}.
 $$
\end{lemma}
\begin{proof}
For each $m\in\textbf{N}$,
let $f_m:\textbf{R}\rightarrow \textbf{R}$ be defined as
$$f_m(x):=
a_{m}+\frac{x-t_m}{t_{m+1}-t_{m}}(a_{m+1}-a_{m}) \ \ \mbox{ for } x\in \textbf{R}.
$$
We show that for $m\in\textbf{N}$ we have
\begin{equation}
\label{nierownoscc}
\begin{array}{ll}
{f}_{m+1}(x) \le {f}_m (x) & \mbox{ for } x\ge t_{m+1}\\
{f}_{m+1}(x) \ge {f}_m (x) & \mbox{ for } x\le t_{m+1}.
\end{array}
\end{equation}
If $x\ge t_{m+1}$ then multiplying both sides of \eqref{nierwonoscwaznaa} by $t_{m+1}-x$ we get
$$
(a_{m+1}-a_{m})\left(1- \frac{x-t_m}{t_{m+1}-t_{m}}\right)+ \frac{x-t_{m+1}}{t_{m+2}-t_{m+1}}(a_{m+2}-a_{m+1})\le0
$$
which is equivalent to $f_{m+1}(x)\le f_m(x).$
Analogously if $x\le t_{m+1}$  we get $f_{m+1}(x)\ge f_m(x).$

We show that the function $g:\textbf{R} \rightarrow \textbf{R}$ defined as
 $$g(r)=\sup\limits_m f_m(r)$$satisfies the requirements  of the Lemma \ref{lemat1} on  $\textbf{R}.$ Indeed, $g$ is convex as  supremum of convex functions. By the first inequality of (\ref{nierownoscc}) and the monotonicity of $\{t_m\}$ we get 
 $$
 f_{m+k}(t_m) \le f_m(t_m) = a_m\ \text{ for } k\ge 1
 .$$
 By the second  inequality of (\ref{nierownoscc}) we get 
 $$
 f_{m-k}(t_m) \le f_m(t_m) \ \text{ for } k=1,\dots, m-1.
 $$
   
This gives that  $g(t_{m})=a_{m}.$ 
\end{proof}

Observe that, we have the following property:
\begin{equation}
\label{55}
g(x)=f_{m}(x)\ \ \ \text{for  }\ t_{m+1}\le x\le t_{m}.
\end{equation}
Indeed, by the first inequality of (\ref{nierownoscc}) and the monotonicity of $\{t_m\}$ we get 
$$
f_m(x) \ge f_{m-1}(x) \ge f_{m-2} (x) \ge \dots f_1(x).
$$
By the second  inequality of (\ref{nierownoscc}) and from the fact that $x\ge t_{m+k}, k\ge 1$ we get 
$$
f_{m+k}(x) \le f_m(x) \ \ \mbox{ for  } k\ge1.
$$

We apply the above lemma to construct a convex function $g$ starting with any  nonnegative converging sequence $\{z_m\}$ with positive limit.

\begin{proposition}
	\label{propozycja22a}
	Let $\{z_m\}\subset \text{R}$ be  a  converging sequence  of nonnegative reals with 
    limit  $z>0.$
	
	There exist a subsequence $\{z_{m_k}\}\subset \{z_m\}$, 
  a sequence $\{t_k\}\subset \textbf{R}$, $t_k\downarrow 0$ and a convex function $g:\textbf{R}\rightarrow \textbf{R},$
	such that
	$$
	g(t_{k})=a_{k},
	$$
	where
	\begin{description}
		\item $a_{k}:=t_{k} z_{m_k}$ if $\{z_{m_k}\}$ is nonincreasing and $0<z<1$,
		\item $a_{k}:=t_{k} z_{m_k}\frac{1}{q}$, $q>z$, if $\{z_{m_k}\}$ is nonincreasing and $z>1$,
		\item  $a_{k}:=-t_{k} z_{m_k}$, if $\{z_{m_k}\}$ is increasing and $z>1$,
		\item  $a_{k}:=-t_{k} z_{m_k}\frac{1}{q}$,  $q<z$, if $\{z_{m_k}\}$ is increasing and $0<z<1$.
	\end{description}
\end{proposition}
\begin{proof}
We can assume that	$z_{m}>0$ for $m\in\textbf{N}$.
	Two cases should be considered:
	\begin{enumerate}
		\item $\{z_{m}\}$ contains an infinite  subsequence $\{z_{m_{k}}\}\subset\{z_{m}\}$,
		$$
		\ \ \ \ \,   \{z_{m_{k}}\}  \mbox{ is nonincreasing, i.e. } z_{m_{k+1}}\le z_{m_{k}},\ k\in\textbf{N},
		$$
		\item $\{z_{m}\}$ contains an infinite subsequence $\{z_{m_{k}}\}\subset\{z_{m}\}$,
		$$
		\{z_{m_{k}}\} \  \mbox{is increasing, i.e. } z_{m_{k+1}}> z_{m_{k}},\ k\in\textbf{N}.
		$$
	\end{enumerate}
	To prove the assertion we show that in both cases  condition \eqref{nierwonoscwaznaa} of Lemma \ref{lemat1} is satisfied.
	
	\noindent
	Case 1.  Without loss of generality we can assume that $z_{m_{k}}=z_{m}$. Moreover, we can assume that
	$\{z_{m}\}$ is decreasing (in case when $\{z_{m}\}$ is constant the existence of the function $g$ satisfying the requirements of the proposition follows trivially from Lemma \ref{lemat1}).
	
	 \noindent
	 Case 1a. Consider first the case when $0<z<1$. By eliminating eventually a finite number of
	  $z_m$ we can assume that  $0<z_m<1$  for $m\in\textbf{N}$.
	
	Let $t_{m}:=(z_{m})^{m}$ and $a_{m}:=(z_{m})^{m+1}$ for $m\in\textbf{N}$.   Both  sequences
	are decreasing.
	Let $m_{1}:=1$, $m_{2}:=2$. Since $\{z_{m}\}$ and $\{a_{m}\}$ are decreasing,
	the number $C$ defined below is positive, i.e.
	$$
	C:=\frac{t_{1}t_{2}}{a_{2}-a_{1}}(z_{2}-z_{1})>0.
	$$
	Since $t_{m}\rightarrow 0$,
	there exists $K\ge m_{2}$ such that
	$$
	t_{m}<C\ \ \mbox{for}\ \  m>K.
	$$
	Put $m_{3}:=K+1$. We have
	$$
	t_{m_{3}}\begin{array}[t]{l}
	<\frac{t_{1}t_{2}}{a_{2}-a_{1}}(z_{2}-z_{1})
	=\frac{t_{1}a_{2}-t_{2}a_{1}+t_{2}a_{2}-t_{2}a_{2}}{a_{2}-a_{1}}
	=\frac{t_{2}(a_{2}-a_{1})-a_{2}(t_{2}-t_{1})}
	{a_{2}-a_{1}}\\
	=t_{2}-\frac{a_{2}(t_{2}-t_{1})}{a_{2}-a_{1}}
	<t_{2}-\frac{a_{2}(t_{2}-t_{1})}{a_{2}-a_{1}}+
	\frac{a_{m_{3}}(t_{2}-t_{1})}{a_{2}-a_{1}}.
	\end{array}
	$$
	The last inequality follows from the fact that $\frac{a_{m_{3}}(t_{2}-t_{1})}{a_{2}-a_{1}}>0$. Thus,
	$$
	t_{m_{3}}-t_{2}<\frac{(a_{m_{3}}-a_{2})(t_{2}-t_{1})}{a_{2}-a_{1}}
	$$
	which  proves that
	$$
	\frac{a_{2}-a_{1}}{t_{2}-t_{1}}\ge\frac{a_{m_{3}}-a_{2}}{t_{m_{3}}-t_{2}}.
	$$
	Assume that we have already defined $m_{1},m_{2},...,m_{k}$, $k\ge 3$. 
	By induction with respect to $k$, starting from
	$m_{k-1}$ and $m_{k}$  we choose $m_{k+1}$ be repeating the reasoning above with $m_{k-1}$ and 
	$m_{k}$ instead $m_{1}$ and $m_{2}$, respectively, and with $K\ge m_{k}$. In this way we choose
	$m_{k+1}$ and we  prove condition \eqref{nierwonoscwaznaa} of Lemma \ref{lemat1}.
	
	\noindent
	Case 1b. When $z>1$ it is enough to take $t_m:=\left( \frac{z_m}{q} \right)^m$ and $a_{m}:=t_{m} z_{m}\frac{1}{q}$ for arbitrary $q>z$ for $m\in\textbf{N}$. Again both sequences are decreasing. Hence, we can repeat
	the above reasoning with these sequences and get the existence of the function $g$ satisfying the requirements.
	\vspace{0.1cm}
	
	\noindent
	Case 2.  As previously, without loss of generality, we can assume that $z_{m_{k}}=z_{m}$. 
	
	\noindent
	Case 2a. Consider the case  $z>1$. By eliminating eventually a finite number of $z_m$ we 
	can assume that $z_m>1$ for $m\in\textbf{N}$.

	Let $t_{m}:=\frac{1}{(z_{m})^{m}}$ and $a_{m}:=-t_{m}z_{m}=-\frac{1}{(z_{m})^{m-1}}$.  
	
	Let $m_{1}:=1$, $m_{2}:=2$. Since $\{t_{m}\}$ is decreasing and $\{z_{m}\}$ is increasing,
	the number $C$ defined below is negative, i.e.
	$$
	C:=\frac{t_{1}t_{2}}{t_{2}-t_{1}}(z_{2}-z_{1})<0.
	$$
	Since $a_{m}\rightarrow 0$
	there exists $K\ge m_{2}$ such that
	$$
	a_{m}>C\ \ \mbox{for}\ \  m>K.
	$$
	Put $m_{3}:=K+1$. We have
	$$
	a_{m_{3}}\begin{array}[t]{l}
	\ge\frac{t_{1}t_{2}}{t_{2}-t_{1}}(z_{2}-z_{1})
	=\frac{t_{2}a_{1}-t_{1}a_{2}+t_{2}a_{2}-t_{2}a_{2}}{t_{2}-t_{1}}
	=\frac{a_{2}(t_{2}-t_{1})-t_{2}(a_{2}-a_{1})}
	{t_{2}-t_{1}}\\
	=a_{2}-\frac{t_{2}(a_{2}-a_{1})}{t_{2}-t_{1}}
	>a_{2}-\frac{t_{2}(a_{2}-a_{1})}{t_{2}-t_{1}}+
	\frac{t_{m_{3}}(a_{2}-a_{1})}{t_{2}-t_{1}}.
	\end{array}
	$$
	The last inequality follows from the fact that
	 $\frac{t_{m_{3}}(a_{2}-a_{1})}{t_{2}-t_{1}}<0$. Thus,
	$$
	a_{m_{3}}-a_{2}>\frac{(t_{m_{3}}-t_{2})(a_{2}-a_{1})}{t_{2}-t_{1}}.
	$$
	As previously, this  proves that
	$$
	\frac{a_{2}-a_{1}}{t_{2}-t_{1}}\ge\frac{a_{m_{3}}-a_{2}}{t_{m_{3}}-t_{2}}.
	$$
	Assume that we have already defined $m_{1},m_{2},...,m_{k}$, $k\ge 3$.
	By induction with respect to $k$, starting from
	$m_{k-1}$ and $m_{k}$  we choose $m_{k+1}$ be repeating the reasoning above with $m_{k-1}$ and 
	$m_{k}$ instead $m_{1}$ and $m_{2}$, respectively, and with $K\ge m_{k}$. In this way we choose
	$m_{k+1}$ and we  prove condition \eqref{nierwonoscwaznaa} of Lemma \ref{lemat1}.
	
      \noindent
      Case 2b. When $0<z<1$ it is enough to take $t_m:=\frac{1}{\left( \frac{z_m}{q} \right)^m}$ for some $q<z$ and $a_m:=-t_m z_{m}\frac{1}{q}.$ We can repeat
	the above reasoning with these sequences and get the existence of the function $g$ satisfying the requirements.

	In consequence, in all cases,  by Lemma \ref{lemat1}, we get the function $g:\textbf{R} \rightarrow \textbf{R}$ defined  by formula  $g(r)=\sup\limits_k f_k(r)$ such that  $g(t_{k})=a_{k},$ where $f_k$, $k\in\textbf{N}$, are as in the proof of Lemma \ref{lemat1}.
	\end{proof}
\vspace{0.3cm}

Proposition \ref{propozycja22a} allows us to formulate  the following corollary  used in the proof of the main result (Theorem \ref{main}).
\vspace{0.3cm}

\begin{proposition}
\label{propozycja22}
Let $Y$ be a normed space and let
 $\{y_m\}\subset Y$ be  such that the sequence  $\{\bar{y}^*(y_m)\}$ converges to  
 $z\in\text{R}\setminus\{0\}$ for 
a certain $\bar{y}^*\in Y^*\setminus\{0\}$. 
 
 There exist a subsequence $\{y_{m_k}\}\subset \{y_m\}$, a functional
 $y^*\in Y^*$ and a convex function $g:\textbf{R}\rightarrow \textbf{R},$
 such that
$$
 g(t_{k})=a_{k},
 $$
 where
 \begin{enumerate}
 \item $t_{k}:=({y}^{*}(y_{m_k}))^{m_k}$, $a_{k}:=t_{k} y^*(y_{m_k}),$ if $\{y^{*}(y_{m_{k}})\}$ is nonicreasing and $0<z<1$,
 \item $t_{k}:=(\frac{{y}^{*}(y_{m_k})}{q})^{m_k}$, $a_{k}:=t_{k} y^*(y_{m_k})\frac{1}{q}$, for some
 $q>z$, if $\{y^{*}(y_{m_{k}})\}$ is nonicreasing and $z>1$,
 \item  $t_{k}:=(\frac{1}{y^{*}(y_{m_k})})^{m_k}$, $a_{k}:=-t_{k} y^*(y_{m_k})$, if $\{y^{*}(y_{m_{k}})\}$ is increasing and  $z>1$,
 \item  $t_{k}:=(\frac{q}{y^{*}(y_{m_k})})^{m_k}$, $a_{k}:=-t_{k} y^*(y_{m_k})\frac{1}{q}$
 for some $q<z$, if $\{y^{*}(y_{m_{k}})\}$ is increasing and  $0<z<1.$
 \end{enumerate}
   \end{proposition}
   
 \begin{proof}
	The proof follows directly from  Proposition \ref{propozycja22a}. Indeed,
	 let us observe that we can always find some $y^*\in Y^*	$, $y^*=c \bar{y}^*$ for a certain $c\neq 0 $ such that $\{y^*(y_m)\}$ is a converging sequence of nonnegative reals
	 with the limit point $z>0$.  
 	 \end{proof}

\section{Main results}
Let $X$  be a vector space and  $Y$ be a real Banach space.
Let   $F: A \rightarrow Y$ be a mapping defined on a subset $A$ of $X$.
\vspace{0.4cm}

\begin{definition}
	\label{pochodna}
	We say that the mapping $F: A \rightarrow Y$ is  directionally differentiable at
	$x_0 \in A $ in the direction $h\neq 0$ such that $x_0+th \in A$ for all  $t$ sufficiently small
	 if  the limit
	
	\begin{equation}
	\label{pochodna-1}
	F'(x_0;h): = \lim_{t\downarrow 0 } \frac{F(x_0+th)- F(x_0)}{t}
	\end{equation}
	exists. The element $F'(x_0;h)$ is called the directional derivative of $F$ at $x_{0}$ in the direction $h.$
\end{definition}

Let $A\subset X$ be a convex subset of $X.$ Let   $F: A \rightarrow Y$ be a $K$-convex mapping on $A$, where $K\subset Y$ is a closed convex
 cone in $Y$.
\vspace{0.4cm}
 By elementary calculations \cite{valadier} one can prove that for any $K$-convex mappings $F$ on $A$, for all $x_0\in A$, $h\in X$, $x_0+th\in A$ for all $t$ sufficiently small, the difference quotient $q(t):=\frac{F(x_0+th)-f(x_0)}{t}$ is nondecreasing
 as a function of $t$. 

Our aim is to prove the following theorem.

\begin{theorem}
	\label{main}
Let $X$ be a  real linear vector space and let $Y$ be a Banach space.
If for every closed convex  cone $K \subset Y$ and every $K$-convex mapping $F: A \rightarrow Y$, $A\subset X$ is a convex subset of $X$, $0\in A,$  the directional derivative $F'(0;h)$ exists for any $h\in X,$ $h\neq 0$ such that $th\in A$ for all $t$ sufficiently small, then $Y$ is weakly sequentially complete. 
\end{theorem}
\begin{proof}
By contradiction, assume that $Y$ is not weakly
sequentially complete.  This means that there exists a nontrivial weak Cauchy sequence $\{y_m\}$ in $Y$, i.e. 
$$
\mbox{for all } y^* \in Y^* \mbox{\ \ the sequence } y^*(y_m) \ \mbox{converges} 
$$
and $\{y_m\}$ is not weakly convergent in $Y$. 
Basing ourselves of the existence of the above sequence we construct a closed convex cone $K\subset Y$ and a $K$-convex mapping  $F: A \rightarrow Y$ such that for a
 given direction $0\neq h\in X$
$$
F'(0;h)\ \ \mbox{does not exist.} 
$$
Let us observe first that we can always find a functional $y^{*}\in Y^{*}\setminus\{0\}$ such that
\begin{equation}
\label{zbiega}
 y^{*}(y_{m})\rightarrow z\in\textbf{R}\setminus\{0\},
 \end{equation}
since otherwise $\{y_{m}\}$  would weakly converge to zero. Let us note that without losing generality
we can assume that $z>0$. Consequently,  by neglecting eventually a finite number of elements we can also assume that
$y^{*}(y_{m})>0$, $m\in\textbf{N}$.

Let $K^{*}\subset Y^{*}$ be defined as
$$
K^{*}:=\cup_{\lambda\ge 0}\lambda y^{*}.
$$
The cone $K^{*}$ is a half-line emanating from $0$ in the direction of $y^{*}$. This is a closed pointed convex cone in $Y^{*}.$

Let
\begin{equation}
\label{eq-cone}
K:=\{y\in Y\ |\  z^{*}(y)\ge 0\ \mbox{for all}\ z^{*}\in K^{*}\}.
\end{equation}
By (\ref{zbiega}), $\{y_{m}\}_{m\in\textbf{N}}\subset K$.
 The cone  $K$ is closed and convex.
By Proposition \ref{propozycja22}, there exist a sequence $t_{k}\downarrow 0 $,
a  subsequence $\{y_{m_k}\}\subset\{y_{m}\}$ and  a convex function $g: \textbf{R}\rightarrow \textbf{R}$ such that 
$g(t_k)=a_k$, where
\begin{enumerate}
 \item $t_{k}:=({y}^{*}(y_{m_k}))^{m_k}$, $a_{k}:=t_{k} y^*(y_{m_k}),$ if $\{y^{*}(y_{m_{k}})\}$ is nonicreasing and $0<z<1$,
 \item $t_{k}:=(\frac{{y}^{*}(y_{m_k})}{q})^{m_k}$, $a_{k}:=t_{k} y^*(y_{m_k})\frac{1}{q}$, where
 $q>z$, if $\{y^{*}(y_{m_{k}})\}$ is nonicreasing and $z>1$,
 \item  $t_{k}:=(\frac{1}{y^{*}(y_{m_k})})^{m_k}$, $a_{k}:=-t_{k} y^*(y_{m_k})$, if $\{y^{*}(y_{m_{k}})\}$ is increasing and  $z>1$,
 \item  $t_{k}:=(\frac{q}{y^{*}(y_{m_k})})^{m_k}$, $a_{k}:=-t_{k} y^*(y_{m_k})\frac{1}{q}$
 for some $q<z$, if $\{y^{*}(y_{m_{k}})\}$ is increasing and  $0<z<1.$
 \end{enumerate}
 Now we construct $K$-convex mappings $F:A \rightarrow Y$ satisfying the requirements of the theorem, where $A:=\{x \in X:\ x=rh,\ r\ge0\}$ for some $h\in X$, $h\neq0$. We limit our attention to the case 1. and the case 3.. The  case 2. and the case 4. require only minor changes. 

Case 1. 
We construct a $K$-convex mapping $F:A\rightarrow Y$ such that 
$$ 
F(t_{k}h)=a_k=t_{k}y_{m_k},\ \ k\in \textbf{N},
$$
where $\{t_k\}$ and $\{y_{m_{k}}\}$ are such that 
$\{t_{k}\}$ and $\{y^{*}(y_{m_{k}})\}$ are as in the case 1. Let $F_k: A \rightarrow Y,$ $k\in \textbf{N},$ be defined as follows.
Let $F_1: A \rightarrow Y$ be defined as
$$F_1(x):=\begin{cases}
y_{m_1}t_{1}+\frac{r-t_{1}}{t_{2}-t_{1}}(y_{m_{2}}t_{2}-y_{m_1}t_{1})&  \ t_{2}< r\\
0 & \  r <t_{2}.\end{cases}$$

For any $k\in \textbf{N}$, $k\ge2$ and any $x=rh$, $r>0$ we define $F_k: A \rightarrow Y$ as follows
$$F_k(x):=\begin{cases}
y_{m_k}t_{k}+\frac{r-t_{k}}{t_{k+1}-t_{k}}(y_{m_{k+1}}t_{k+1}-y_{m_k}t_{k})&  \ t_{k+1}< r\le t_{k}\\
0 & \  r \notin (t_{k+1}, t_{k}].\end{cases}$$

We start by showing that the mapping  $F: A \rightarrow Y$ defined as $F(x):=\sum\limits_{i=1}^\infty F_i(x)$
 is well defined. To this aim it is enough to observe that for any $x\in A$ with $x=rh$, where $t_{k+1}<r\le t_k$,  for any $N\in \textbf{N}$ we have
 $$
 \sum\limits_{i=1}^N F_i(x)=\begin{cases}
 F_k(x) & N \ge k\\
 0 & N<k.
 \end{cases}
 $$
 Consequently, for $t_{k+1}<r\le t_k$, we have
\begin{equation}
\label{funkcja}
F(x)=\sum\limits_{i=1}^{\infty} F_i(x)=\lim\limits_{N\rightarrow\infty}\sum\limits_{i=1}^N F_i(x)=F_{k}(x).
\end{equation}
It is easy to see that for $r>t_2$ we have $F(x)=F_1(x).$ Moreover, $F(0)=0.$
We show that the mapping  $F$ is $K$-convex on $A$ with respect to cone $K$ defined by (\ref{eq-cone}).
To this aim we use Lemma \ref{lemma-convex},

Precisely, we start by showing that for  $x\in A$,  $x=rh$, $r>0$ we have 
$$
y^*(F(x))=g(r),
$$
where function $g$ is as in Proposition \ref{propozycja22}.
Indeed, if $r\in (t_{{k+1}}, t_{k}]$ for some $k \ge 2$,
then by (\ref{55}) and (\ref{funkcja}), 
$$y^*(F(x))=y^*(F_k(rh)) =y^*\left(y_{m_k}t_{k}+\frac{r-t_{k}}{t_{k+1}-t_{k}}(y_{m_{k+1}}t_{k+1}-y_{m_k}t_{k})  \right) =f_{k}(r)=g(r).$$
If $r>t_2$, then 
$$
y^*(F(x))=y^*(F_1(rh)) =y^*\left(y_{m_1}t_{1}+\frac{r-t_{1}}{t_{2}-t_{1}}(y_{m_{2}}t_{2}-y_{m_1}t_{1})  \right) =f_{1}(r)=g(r).
$$

 
Now, take any  $z^*\in K^*.$ By the definition of $K^*$, $z^*=\beta y^*$ for some $\beta \ge 0$ and by convexity of $g$
we get
\begin{equation}
\label{02-10}
\begin{array}{l}
z^{*}(F\left((\lambda a_{1}+(1-\lambda)a_{2})h\right))
=\beta y^*(F\left((\lambda a_{1}+(1-\lambda)a_{2})h\right))\\=\beta g(\lambda a_1+(1-\lambda )a_2) \le \beta\lambda  g(a_1) + \beta(1-\lambda) g(a_2)=\\
=\lambda z^*( F(a_1 h)) + (1-\lambda )z^*( F(a_2 h )).
\end{array}
\end{equation}
By Lemma  \ref{lemma-convex}, this proves that $F$ is $K$-convex on $A$.

Case 2. We take $t_{k}:=(\frac{{y}^{*}(y_{m_k})}{q})^{m_k}$ and $a_k:= t_k y^*(y_{m_k})\frac{1}{q},$ where $q>z.$ With this choice $\{t_k\}$, $\{a_k\}$, $k \in \textbf{N},$ by repeating the reasoning from Case 1. we get the required mapping $F.$

Case 3. We construct a $K$-convex mapping $F:A \rightarrow Y$ such that 
$$ 
F(t_{k}h)=a_k=-t_{k}y_{m_k},\ \ k\in \textbf{N},
$$
where $\{t_k\}$ and $\{y_{m_k}\}$ are such that $\{t_k\}$ and $\{y^*(y_{m_k})\}$ are as in the case 3. 


Let $F_k: A \rightarrow Y$, $k\in \textbf{N},$ be defined as follows.
Let  $F_1: A\rightarrow Y$ be defined as
$$F_1(x):=\begin{cases}
-y_{m_1}t_{1}-\frac{r-t_{1}}{t_{2}-t_{1}}(y_{m_{2}}t_{2}-y_{m_1}t_{1})&  \ t_{2}< r\\
0 & \  r <t_{2}.\end{cases}$$

For any $k\in \textbf{N}$, $k\ge2$   and any $x\in A$, $x=rh$, $r>0$ we define $F_k: A \rightarrow Y$ as follows
$$F_k(x):=\begin{cases}
-y_{m_k}t_{k}-\frac{r-t_{k}}{t_{k+1}-t_{k}}(y_{m_{k+1}}t_{k+1}-y_{m_k}t_{k})&  \ t_{k+1}< r\le t_{k}\\
0 & \  r \notin (t_{k+1}, t_{k}].\end{cases}$$

Now the mapping $F:A \rightarrow Y$ is defined as in Case 1. by  formula (\ref{funkcja}). As previously, for $r>t_2$ we have $F(x)=F_1(x)$ and $F(0)=0.$ For any $z^*\in K^*,$ the convexity of function $z^*(F)$  is proved  by the same arguments as in Case 1. in the proof of  (\ref{02-10}).

Case 4. We have $t_{k}:=(\frac{q}{y^{*}(y_{m_k})})^{m_k}$ and $a_{k}:=-t_{k} y^*(y_{m_k})\frac{1}{q}$
 for some $q<z.$ With this choice $\{t_k\}$, $\{a_k\}$, $k \in \textbf{N},$ by repeating the reasoning from Case 3. we get the required mapping $F.$

In all the cases, the directional derivative of $F$ at zero in the direction $h\in X$  equals
$$
F'(0;h)=\lim_{t\downarrow 0}\frac{F(th)}{t}.
$$
Summing up the considerations above we get the following formulas: 
$$
\begin{array}{ll}
\mbox{ Case 1. \ \ } F'(0;h)=\lim\limits_{k\rightarrow +\infty}\frac{F(t_{k}h)}{t_{k}}=\lim\limits_{k\rightarrow +\infty}y_{m_k},\\ 
\mbox{ Case 2. \ \ } F'(0;h)=\lim\limits_{k\rightarrow +\infty}\frac{F(t_{k}h)}{t_{k}}=\lim\limits_{k\rightarrow +\infty}\frac{1}{q}y_{m_k}\ \ \mbox{ for some }q>z.\\
\mbox{ Case 3. \ \ } F'(0;h)=\lim\limits_{k\rightarrow +\infty}\frac{F(t_{k}h)}{t_{k}}=\lim\limits_{k\rightarrow +\infty}-y_{m_k}.\\ 
\mbox{ Case 4. \ \ } F'(0;h)=\lim\limits_{k\rightarrow +\infty}\frac{F(t_{k}h)}{t_{k}}=\lim\limits_{k\rightarrow +\infty}-\frac{1}{q}y_{m_k}\ \ \mbox{ for some }q<z,\\

\end{array}
$$
for the respective sequences $t_{k}\downarrow 0$ and $\{y_{m_k}\}.$

On the other hand, by Proposition \ref{podciag}, the subsequences
 $\{y_{m_k}\}$ appearing in the above formulas are not weakly convergent. 
 
Hence, we constructed a cone $K$ and a $K$-convex mapping $F$ on $A$
which is not directionally differentiable at $0$ in the direction $h,$ which completes the proof.
\end{proof}
\section{Extension of mapping $F$}

When  $X$ is a Hilbert space  we can extend the mapping $F$ from Theorem \ref{main} to the half-space $H:=\{x \in X: \langle x, h \rangle \ge 0\},$ where $h$ is as in Theorem \ref{main}.

\begin{proposition}
\label{kolacja}
Let $X$ be a Hilbert space and $h\in X\setminus\{0\}.$ Let set $A,$ mapping $F:A \rightarrow Y,$ and cone $K$ be defined as in the proof of Theorem \ref{main}. Then there exists a convex extension $\bar{F}: H \rightarrow Y$ of  $F$ to the half-space $H=\{x\in X\ :\ \langle x, h \rangle \ge 0\}.$
\end{proposition}
\begin{proof}
Let $h\neq 0,$ set $A$, cone $K,$ and mapping $F$ be defined as in the proof of Theorem \ref{main}. 
Let $\bar{F}: H \rightarrow Y$ be defined as follows
$$\bar{F}(x)=  F(\langle x, h \rangle h).$$
Let $x_1, x_2\in H$ and $\lambda \in [0,1].$
$K$-convexity of mapping $\bar{F}$ on $H$ follows from $K$-convexity of mapping $F$ on $A,$  
$$
\begin{array}{ll}
\bar{F}(\lambda x_1 + (1-\lambda)x_2)={F}(\langle \lambda x_1 + (1-\lambda)x_2, h \rangle h)=F(\lambda \langle x_1, h \rangle h  + (1-\lambda)\langle x_2, h \rangle h) \\
\le_K \lambda F(\langle x_1,h\rangle) + (1-\lambda) F(\langle x_2, h\rangle h)=\lambda\bar{F}(x_1) + (1-\lambda)\bar{F}(x_2).
\end{array}
$$
\end{proof}
\newpage

In view of the above Proposition we can formulate the following theorem.
\begin{theorem}
Let $X$ be a  Hilbert space and let $Y$ be a Banach space.
If for every closed convex  cone $K \subset Y,$ and every $K$-convex mapping $F: H \rightarrow Y$ on $H=\{x\in X\ :\ \langle x,h \rangle \ge 0\}$ for all $h\in X$, $\|h\|=1$  the directional derivative $F'(0;h)$ exists,  
 then $Y$ is weakly sequentially complete. 

\end{theorem}
\begin{proof}
Without loss of generality, in Theorem \ref{main}, we can define $K$-convex mapping $F$ on $A=\{x\in X: x=rh,\ r\ge0\},$ where direction $h\in X$ is such that $\|h\|=1.$
From
Proposition \ref{kolacja} we can extend the $K$-convex mapping $F$  on $A$ to the $K$-convex mapping $\bar{F}$ on $H.$ Since, $\bar{F}(0)=0$ the directional derivative of $\bar{F}$ at $0$ in the direction $h$ is equal to
$$
\lim\limits_{t \downarrow 0} \frac{\bar{F}(\langle th,h\rangle h)}{t}=\lim\limits_{t \downarrow 0} \frac{\bar{F}(t\|h\|^2 h)}{t}=\lim\limits_{t \downarrow 0}\frac{F(t h)}{t}.
$$

The directional derivative of $F$ at zero in the direction $h\in X$ does not exist which completes the proof.
\end{proof}

\section{Comments}
In \cite{valadier} directional derivatives of cone convex mappings are defined via the concept of infimum of the set $\left\{\frac{F(x_0+th)-F(x_0)}{t}:\ t>0\right\}$. Let $K$ be a pointed cone.  An element $a$ is the infimum  of $A$ if $a$ is a lower bound, i.e. $a\le_K x\  \forall x \in A$ and, for any lower bound $a'$ of $A$, we have $a'\le_K a$ (c.f. \cite{bourbaki} Ch.7, par.1, nr 8, \cite{as} Ch.2, p.18).
 Without additional assumptions about cone $K$ one cannot expect the equality
\begin{equation}
\label{rownosccc}
\lim\limits_{t\downarrow 0 } \frac{F(x_0+th)-F(x_0)}{t}=\inf\left\{\frac{F(x_0+th)-F(x_0)}{t}:\ t>0\right\}
.\end{equation}

In \cite{as} we can find the following proposition.
\begin{proposition}
Let $(X,\tau)$ be a Hausdorff topological vector space ordered by the closed convex pointed cone $K$. If the net $(x_i)_{i\in I}\subset X$ is nonincreasing and convergent to $x\in X,$ then $\{x_i\ : \ i\in I \}$ is bounded below and $x=\inf\{x_i\ :\ i\in I\}.$
\end{proposition}

As a corollary we get

\begin{proposition}If there exists $F'(x_0;h)$,  there exists $$\inf\left\{\frac{F(x_0+th)-F(x_0)}{t}:\ t>0\right\}
.$$ 
\end{proposition}
On the other hand, in \cite{peressini} we can find the following example. Let $X=l^\infty$ and let us consider partial order generated by the cone $l^\infty_+:=\{x\in l^\infty\ :\ x^k\ge 0 \ \forall k\ge 1\}$, $l^\infty_+$ is a pointed closed convex cone. The sequence $\{x_n\}\subset l^\infty$ defined (for $n$ fixed) $x_n^k:=\begin{cases} -1 & \mbox{if }1\le k\le n\\
0 & \mbox{if }   k>n,
\end{cases}$
is nonincreasing, and $\inf\{x_n\ :\ n\ge 1\}=(-1,-1,-1,\dots).$ But $\{x_n\}$ does not converge to its infimum.

We prove the cone convexity of the mapping $F$ constructed in the  proof of Theorem  \ref{main} (and in the proof of Theorem 6.2) for the cone $K$
$$
K=\{y \in Y| \ \ z^*(y)\ge 0\ \mbox{ for all } z^* \in K^*\}.
$$
This cone is not pointed, so the concept of the infimum of the set is not defined. Let us note that the concept of directional derivative introduced in Definition \ref{pochodna} is well defined independently on whether cone $K$ is pointed or not.

\end{document}